\documentclass[12pt]{article}
\font\smallit=cmti10

\usepackage{amssymb,latexsym,amsmath,epsfig,amsthm}
\usepackage{rotating}
\usepackage{graphicx}
\usepackage{amssymb}
\usepackage{lineno}
\usepackage{enumitem}
\usepackage{cite}
\usepackage[top=3cm, bottom=3cm, left=3.5cm, right=3.5cm]{geometry}
\makeatletter
\renewcommand\section{\@startsection {section}{1}{\z@}
	{-30pt \@plus -1ex \@minus -.2ex}
	{2.3ex \@plus.2ex}
	{\normalfont\normalsize\bfseries\boldmath}}

\renewcommand\subsection{\@startsection{subsection}{2}{\z@}
	{-3.25ex\@plus -1ex \@minus -.2ex}
	{1.5ex \@plus .2ex}
	{\normalfont\normalsize\bfseries\boldmath}}
\renewcommand{\@seccntformat}[1]{\csname the#1\endcsname. }
\makeatother
\newtheorem{theorem}{Theorem}
\newtheorem{lemma}{Lemma}
\newtheorem{proposition}{Proposition}

\theoremstyle{definition}

\begin{document}
	
	\begin{center}
		\textbf{\large An improved threshold for the number of distinct intersections of intersecting families} 
		\vskip 20pt
		\textbf{Jagannath Bhanja, Sayan Goswami}\\
		{\smallit The Institute of Mathematical Sciences, A CI of Homi Bhabha National Institute, C.I.T. Campus, Taramani, Chennai 600113, India}\\
		{\tt jbhanja@imsc.res.in, sayangoswami@imsc.res.in}
	\end{center}
	\vskip 20pt

	\begin{abstract}
	 A family $\mathcal{F}$ of subsets of $\{1,2,\ldots,n\}$ is called a $t$-intersecting family if $|F\cap G| \geq t$ for any two members $F, G \in \mathcal{F}$ and for some positive integer $t$. If $t=1$, then we call the family $\mathcal{F}$ to be intersecting. Define the set $\mathcal{I}(\mathcal{F}) = \{F\cap G: F, G \in \mathcal{F} \text{ and } F \neq G\}$ to be the collection of all distinct intersections of $\mathcal{F}$. Frankl et al. proved an upper bound for the size of $\mathcal{I}(\mathcal{F})$ of intersecting families $\mathcal{F}$ of $k$-subsets of $\{1,2,\ldots,n\}$. Their theorem holds for integers $n \geq 50 k^2$. In this article, we prove an upper bound for the size of $\mathcal{I}(\mathcal{F})$ of $t$-intersecting families $\mathcal{F}$, provided that $n$ exceeds a certain number $f(k,t)$. Along the way we also improve the threshold $k^2$ to $k^{3/2+o(1)}$ for the intersecting families.   
	\end{abstract}
	
	\vspace{10pt}
	
	\noindent {\bf 2020 Mathematics Subject Classification:} 05D05
	
	\vspace{10pt}
	\noindent {\bf Keywords:} $t$-intersecting family, set intersection, Erd\H{o}s-Ko-Rado theorem

	\section{Introduction}
	We denote the standard $n$-element set $\{1,2,\ldots,n\}$ by $[n]$, the set of all subsets of $[n]$ by $2^{[n]}$, and for $0 \leq k \leq n$ the collection of all $k$-element subsets of $[n]$ by $\binom{[n]}{k}$. We use the standard notation $|S|$ for the cardinality of a set $S$. We also use the following standard analytical notations. For non-negative functions $f,g$, we write $g = \mathcal{O}(h)$ to mean that $g(x) \leq c \cdot h(x)$ for some positive constant $c$; $g = o(h)$ to mean that $\frac{g(x)}{h(x)} \to 0$; $g = \Theta(h)$ to mean that $g=\mathcal{O}(h)$ and $h=\mathcal{O}(g)$.
	
	We call a family $\mathcal{F} \subset 2^{[n]}$ to be intersecting if the intersection of any two members of $\mathcal{F}$ is non-empty, and we call $\mathcal{F}$ to be $t$-intersecting if $|F\cap G| \geq t$ for any two members $F, G \in \mathcal{F}$ and for some positive integer $t$. We call a family to be the complete sunflower if every subset which contain a fix $t$-set (say) $X$ is a member of that family; we denote such a family by $S_X$. Any subset of $S_X$ is called Sunflower.
	
	The Erd\H{o}s-Ko-Rado theorem \cite{erdos} is a pioneer result in extremal combinatorics. 
    
    \begin{theorem} [Erd\H{o}s-Ko-Rado theorem] \label{ekr-t-theorem} There exists some $n_0(k, t)$ such that if $n \geq n_0(k, t)$ and $\mathcal{F} \subset \binom{[n]}{k}$ is $t$-intersecting, then  
    	\[|\mathcal{F}| \leq \binom{n-t}{k-t}.\]
    \end{theorem}

    Frankl \cite{frankl78} proved that the Erd\H{o}s-Ko-Rado theorem holds for $n_0(k, t)=(t+1)(k-t+1)$ and $t\geq 15$. Wilson \cite{wilson} proved the theorem holds with same $n_0(k, t)$ and for all $t$. The bound in the Erd\H{o}s-Ko-Rado theorem is achieved for the complete sunflower $\mathcal{F}$.
    
    Let $\mathcal{I}(\mathcal{F}) = \{F\cap G: F, G \in \mathcal{F} \text{ and } F \neq G\}$ be the collection of all distinct intersections of $\mathcal{F}$. In \cite{frankl22} Frankl, Kiselev, and Kupavskii proved the following theorem. 
           
    \begin{theorem} [Frankl, Kiselev, and Kupavskii] \label{frankl-theorem} Suppose that $\mathcal{F} \subset \binom{[n]}{k}$ is intersecting with $k \geq 2$ and $n \geq 50k^2$. Then    
    	\[|\mathcal{I}(\mathcal{F})| \leq |\mathcal{I}(\mathcal{A})|,\]
    where 
    \[\mathcal{A} = \left\{ A \in \binom{[n]}{k}: |A \cap \{1,2,3\}| \geq 2\right\}.\]
    \end{theorem}
    
    In this paper, we prove that a similar result also holds for $t$-intersecting families $\mathcal{F}$ of $\binom{[n]}{k}$. More precisely, we prove the following theorem. 
    
    \begin{theorem} \label{main-theorem} There exists some number $f(k, t)$ depending on $k, t$ such that if $n \geq f(k, t)$ and $\mathcal{F} \subset \binom{[n]}{k}$ is $t$-intersecting, then     
    	\[|\mathcal{I}(\mathcal{F})| \leq |\mathcal{I}(\mathcal{A}_t)|,\]
    	where 
    	\[\mathcal{A}_t = \left\{ A \in \binom{[n]}{k}: |A \cap \{1,2,\ldots,t+2\}| \geq t+1\right\},\] 
    and  $f(k, t)$ satisfies the following property:
    \begin{enumerate}
       
       \item if $t=\mathcal{O}(1)$, then $f(k,t)=\Theta (k^{3/2+\epsilon})$, where $\epsilon= \frac{10+t}{2(k-t-2)}$,

       \item if $t=o(k)$, then
            
            \begin{enumerate}

               \item if $t=\Theta(k^{\epsilon^\prime})$ with $\epsilon^\prime \leq 1/4$, then $f(k,t)=\Theta (k^{3/2+\epsilon^\prime})$,
               
               \item if $t=\Theta(k^{\epsilon^\prime})$ with $\epsilon^\prime > 1/4$, then $f(k,t)=\Theta (k^{1+2 \epsilon^\prime})$,

           \end{enumerate}

       \item if $t=\Theta (k)$, then $f(k,t)=\Theta(k^3)$.

    \end{enumerate}

    \end{theorem}

    \section{Preliminaries}   
    We need certain notions which were introduced in \cite{frankl22}; apart from that we also define some analogous notions for $t$-intersecting families.  
    
    For a family $\mathcal{F}$, we call  
    \[\mathcal{T}(\mathcal{F}) := \{T \in 2^{[n]}: |T| \leq k, |T \cap F| \geq t \text{ for all } F\in \mathcal{F}\}\]
    to be the family of $t$-transversals. Then $\mathcal{F} \subset \mathcal{T}(\mathcal{F})$ if and only if $\mathcal{F}$ is $t$-intersecting. For a non-negative integer $\ell \leq n$ we call $\mathcal{F}^{(\ell)} = \{F \in \mathcal{F}, |F|=\ell\}$ the $\ell$-th level of the family $\mathcal{F}$, and set $\mathcal{F}^{(\leq \ell)} = \bigcup_{i=1}^{\ell} \mathcal{F}^{(\ell)}$. We call a $t$-intersecting family $\mathcal{F}$ to be saturated if the property that $|F \cap G| \geq t$ ceases to hold for addition of any new member to $\mathcal{F}$. It is easy to observe that, an $t$-intersecting family $\mathcal{F} \subset \binom{[n]}{k}$ is saturated if and only if $\mathcal{F} = \mathcal{T}(\mathcal{F})^{(k)}$. So, in the rest of this paper we assume that $\mathcal{F}$ is $t$-intersecting and saturated. A family $\mathcal{B}$ is an antichain if for any $B, B^\prime \in \mathcal{B}$ with $B \subset B^\prime$, then $B=B^\prime$. 
     
    To prove Theorem \ref{main-theorem} we need certain lemmas, which are just the $t$-intersecting analogue of the lemmas proved in \cite{frankl22} (see Lemma 1.3, Lemma 2.2, and Lemma 2.3 of \cite{frankl22}). All the lemmas in this section can be proved by exploiting the following fact: in a $t$-intersecting family $|F \cap G|<t$ works as disjoint sets $F, G$ of an intersecting family. So, we leave the details of the proof to the reader.    
     
    \begin{lemma}\label{antichain-lemma}
    	Let $\mathcal{F} \subset \binom{[n]}{k}$ be a saturated $t$-intersecting family. Let $\mathcal{B}=\mathcal{B}(\mathcal{F})$ be the family of minimal sets in $\mathcal{T}(\mathcal{F})$. Then 
    	\begin{enumerate}
    		\item $\mathcal{B}$ is a $t$-intersecting antichain,    		
    		\item $\mathcal{F} = \left\{D \in \binom{[n]}{k}: \exists B \in \mathcal{B}, B \subset D \right\}$,
    		\item $\mathcal{B}$ contains no sunflower of size $k+1$. 
    	\end{enumerate}
    \end{lemma}

    \begin{lemma}\label{triangle-lemma}
    	Let $\mathcal{F} \subset \binom{[n]}{2}$ be a $t$-intersecting family. Then $\mathcal{F}$ is either a sunflower or a $(t+2)$-triangle of the form $\{\{1,2,\ldots ,t+1\}, \{2,3,\ldots ,t+2\},\{\{1,t+2\}\cup D : D\subset \{2,3,\ldots ,t+1\}, \vert D\vert =t-1\}\}$.
    \end{lemma}

    To state the following lemma we need some further notions. We denote $s=s(\mathcal{B})$ for $\min\{|B|: B \in \mathcal{B}\}$ and the $t$-covering number $\tau(\mathcal{B})$ for $\min \{|T|: |T \cap B| \geq t \text{ for all } B \in \mathcal{B} \}$. 
    
    \begin{lemma}\label{main-lemma}
    	Let $\ell$ be an integer such that $t+1 \leq \ell \leq k$. Suppose that $\mathcal{F} \subset \binom{[n]}{k}$ is a saturated $t$-intersecting family. Assume that $\mathcal{B}=\mathcal{B}(\mathcal{F})$, $s \geq t+1$, and $\tau(\mathcal{B}^{(\leq \ell)}) \geq t+1$. Then 
    	\begin{equation}\label{bl-equation}
    		|\mathcal{B}^{(\ell)}| 
    		\leq s \cdot \ell \cdot (k-t+1)^{\ell-t-1}.
    	\end{equation}
    \end{lemma}

    \section{The proof of Theorem \ref{main-theorem}} 
    
     We begin with estimating the exact size of $\mathcal{I}(\mathcal{A}_t)$. 
    
    \begin{proposition} \label{proposition} We have  
    	\begin{align}\label{at-eqn}
    		&|\mathcal{I}(\mathcal{A}_t)|\\
    		&= \binom{t+2}{t} \sum_{j=0}^{k-t-1} \binom{n-t-2}{j} + \binom{t+2}{t+1} \sum_{j=0}^{k-t-2} \binom{n-t-2}{j} 
    		+ \sum_{j=0}^{k-t-3} \binom{n-t-2}{j}. \nonumber
    	\end{align}  
    \end{proposition}
    
    \begin{proof}
    	Let $A, A^\prime \in \mathcal{A}_t$. As $|A \cap A^\prime \cap \{1,2,\ldots,t+2\}| \geq t$, one has total $\binom{t+2}{t}+\binom{t+2}{t+1}+1$ possibilities for $|A \cap A^\prime \cap \{1,2,\ldots,t+2\}|$. Consider one of such possibilities $A \cap A^\prime \cap \{1,2,\ldots,t+2\} = \{1,2,\ldots,t\}$. Then both $A \cap \{1,2,\ldots,t+2\}$ and $A^\prime \cap \{1,2,\ldots,t+2\}$ are of the form $\{1,2,\ldots,t+2\}\setminus\{x\}$ for some $x \in [t+2]$. As $n$ large enough $A \cap A^\prime = \{1,2,\ldots,t\} \cup B$ for some $B \subset \{t+3,t+4,\ldots,n\}$. Further, as $A \neq A^\prime$ we must have $|B| \leq k-t-1$. Thus, in this particular case there are total $\sum_{i=0}^{k-t-1} \binom{n-t-2}{i}$ possible values for $|A \cap A^\prime|$. Considering all the possibilities for $A \cap A^\prime \cap \{1,2,\ldots,t+2\}$ we obtain
    	\begin{align*}
    		&|\mathcal{I}(\mathcal{A}_t)|\\ 
    		&= \binom{t+2}{t} \sum_{j=0}^{k-t-1} \binom{n-t-2}{j} + \binom{t+2}{t+1} \sum_{j=0}^{k-t-2} \binom{n-t-2}{j} + \sum_{j=0}^{k-t-3} \binom{n-t-2}{j}.
    	\end{align*}     
    \end{proof}

\begin{proof}[Proof of Theorem \ref{main-theorem}]
    Let $k=t+1$. Then by Lemma \ref{triangle-lemma} the family $\mathcal{F}$ is either a $(t+2)$-triangle or a sunflower. In both the cases the theorem is trivial. Thus, we may assume that $k \geq t+2$.
    
    For a $t$-element subset $X$ of $[n]$, let $S_X$ be the complete sunflower. Then
    \begin{align}\label{sx-eqn}
    	|\mathcal{I}(S_X)| 
    	&= \sum_{j=0}^{k-t-1} \binom{n-t}{j} \nonumber\\
    	&= 2\sum_{j=0}^{k-t-1} \binom{n-t-1}{j} - \binom{n-t-1}{k-t-1}\nonumber\\
    	&= 4\sum_{j=0}^{k-t-1} \binom{n-t-2}{j} - 2\binom{n-t-2}{k-t-1} - \binom{n-t-1}{k-t-1}\nonumber\\
    	&= 2\binom{n-t-2}{k-t-1} + 4\sum_{j=0}^{k-t-2} \binom{n-t-2}{j} - \binom{n-t-1}{k-t-1}.
    \end{align}
   Comparing Eq. (\ref{at-eqn}) and Eq. (\ref{sx-eqn}) one can see that $|\mathcal{I}(S_X)| < |\mathcal{I}(\mathcal{A}_t)|$. Thus, we may suppose that $\mathcal{F}$ is not a complete sunflower. This implies $\mathcal{B}^{(t)}=\emptyset$. 
    
    Now, we partition $\mathcal{F}$ as $\mathcal{F} = \mathcal{F}^{(s)} \cup \cdots \cup \mathcal{F}^{(k)}$, where $s=s(\mathcal{B(\mathcal{F})})$ and the subfamilies $\mathcal{F}^{(\ell)}$'s are defined as follows: $F \in \mathcal{F}^{(\ell)}$ if $\ell = \max\{|B|: B \in \mathcal{B}, B \subset F\}$. Set $\mathcal{I}^\ell = \{F \cap G: F \in \mathcal{F}^{(\ell)}, G \in \mathcal{F}^{(s)}\cup \cdots \cup \mathcal{F}^{(\ell)}\}$. Then 
    \[|\mathcal{I}(\mathcal{F})| \leq \sum_{\ell=s}^{k} |\mathcal{I}_\ell|\] 
    
    To calculate $|\mathcal{I}_\ell|$ we recall that, for every $F \in \mathcal{F}^{(\ell)}$ there exist $B \in \mathcal{B}^{(\ell)}$ such that $B \subset F$. Then for any $F^\prime \in \mathcal{F}$, we have
    \[F \cap F^\prime = (B\cap F^\prime) \cup ((F\setminus B)\cap F^\prime).\] 
    Here there are at most $\sum_{i=t}^{\ell} \binom{\ell}{i}$ possibilities for $B\cap F^\prime$ and $(F\setminus B)\cap F^\prime$ can be any subset of $[n]$ of size $k-\ell$. Thus, using Eq. (\ref{bl-equation}) we get
    \begin{align}\label{il-eqn}
    	|\mathcal{I}_\ell| 
    	&\leq |\mathcal{B}^{(\ell)}| \sum_{i=t}^{\ell} \binom{\ell}{i} \sum_{j=0}^{k-\ell} \binom{n}{j} \nonumber\\
    	& < s\ell(k-t+1)^{\ell-t-1} \sum_{i=t}^{\ell} \binom{\ell}{i} \sum_{j=0}^{k-\ell} \binom{n}{j}.
    \end{align}
 
    If $\tau(\mathcal{B}^{(t+1)})=t+1$, then by Lemma \ref{triangle-lemma} we obtain that $\mathcal{F}$ is a triangle. In this case, Theorem \ref{main-theorem} is trivial. Let $\alpha$ be the smallest integer such that $\tau(\mathcal{B}^{(\leq \alpha)}) \geq t+1$. Thus $\alpha \geq t+2$. This implies that the family $\cup_{i=1}^{\alpha-1} \mathcal{F}^{(i)}$ is a sunflower and hence 
    \begin{equation}\label{lessl-eqn}
    	|\bigcup_{i=1}^{\alpha-1} \mathcal{I}_i| 
    	\leq |\mathcal{I}(S_X)| 
    	= 2\binom{n-t-2}{k-t-1} + 4\sum_{j=0}^{k-t-2} \binom{n-t-2}{j} - \binom{n-t-1}{k-t-1}.
    \end{equation}
    On the other hand, by Eq. (\ref{il-eqn}) we have  
    \begin{align}\label{biggerl-eqn}
    	|\bigcup_{\ell=\alpha}^{k} \mathcal{I}_\ell| 
    	&\leq \sum_{\ell=\alpha}^{k} |\mathcal{I}_\ell|\nonumber\\
    	&< \sum_{\ell=\alpha}^{k} \ell^2(k-t+1)^{\ell-t-1} \sum_{i=t}^{\ell} \binom{\ell}{i} \sum_{j=0}^{k-\ell} \binom{n}{j}\nonumber\\
    	&\leq \sum_{\ell=t+2}^{k} \ell^2(k-t+1)^{\ell-t-1} \sum_{i=t}^{\ell} \binom{\ell}{i} \sum_{j=0}^{k-\ell} \binom{n}{j}.
    \end{align}
    Adding Eq. (\ref{lessl-eqn}) and Eq. (\ref{biggerl-eqn}) we obtain
    \begin{align}\label{total-eqn}
    	|\mathcal{I}(\mathcal{F})| 
    	&< 2\binom{n-t-2}{k-t-1} + 4\sum_{j=0}^{k-t-2} \binom{n-t-2}{j} - \binom{n-t-1}{k-t-1}\nonumber\\
    	&\quad +\sum_{\ell=t+2}^{k} \ell^2(k-t+1)^{\ell-t-1} \sum_{i=t}^{\ell} \binom{\ell}{i} \sum_{j=0}^{k-\ell} \binom{n}{j}.
    \end{align}
    We claim that there exist some $f(k,t)$ such that for $n \geq f(k,t)$ we have
    \begin{align*}
        &2\binom{n-t-2}{k-t-1} + 4\sum_{j=0}^{k-t-2} \binom{n-t-2}{j} - \binom{n-t-1}{k-t-1}\\
    	&\quad +\sum_{\ell=t+2}^{k} \ell^2(k-t+1)^{\ell-t-1} \sum_{i=t}^{\ell} \binom{\ell}{i} \sum_{j=0}^{k-\ell} \binom{n}{j}r\\
    	&\leq \binom{t+2}{t} \binom{n-t-2}{k-t-1} + \binom{t+3}{t+1} \sum_{j=0}^{k-t-2} \binom{n-t-2}{j} + \sum_{j=0}^{k-t-3} \binom{n-t-2}{j}\\
    	&= \binom{t+2}{t} \sum_{j=0}^{k-t-1} \binom{n-t-2}{j} + \binom{t+2}{t+1} \sum_{j=0}^{k-t-2} \binom{n-t-2}{j}	
    	+ \sum_{j=0}^{k-t-3} \binom{n-t-2}{j}.
    \end{align*}
Adjusting similar terms equivalently we claim that, for sufficiently large $n$ we have 
\begin{align}\label{threshold-eqn}
	&\sum_{\ell=t+2}^{k} \ell^2(k-t+1)^{\ell-t-1} \sum_{i=t}^{\ell} \binom{\ell}{i} \sum_{j=0}^{k-\ell} \binom{n}{j}\nonumber\\
	&\leq \binom{n-t-1}{k-t-1} + \frac{(t+2)(t+1)-4}{2} \cdot \binom{n-t-2}{k-t-1} \nonumber\\
	&\quad + \frac{(t+3)(t+2)-8}{2} \cdot \sum_{j=0}^{k-t-2} \binom{n-t-2}{j} 
	+ \sum_{j=0}^{k-t-3} \binom{n-t-2}{j}.
\end{align}
We calculate the order of $n$ (if written in the form of power of $k$) for which this inequality holds. 



Suppose that $n=\Theta(k^{s})$ for some real number $s$. Then the growth of $\sum_{j=0}^{k-\ell} \binom{n}{j}$ is $k^{s(k-t-2)}$ which we obtain by putting $\ell=t+2$ in $\sum_{j=0}^{k-\ell} \binom{n}{j}$. On the other hand, the growth of $\sum_{\ell=t+2}^{k} \ell^2(k-t+1)^{\ell-t-1} \sum_{i=t}^{\ell} \binom{\ell}{i}$ is $k^{3k/2-t+2}$ which we obtain by putting $\ell=k$ in $\sum_{\ell=t+2}^{k} \ell^2(k-t+1)^{\ell-t-1} \sum_{i=t}^{\ell} \binom{\ell}{i}$. If $\epsilon \geq \frac{10+t}{2(k-t-2)}$, then $\frac{3}{2} + \epsilon \geq \frac{{3k/2-t+2}}{(k-t-2)}$. Therefore, for any real $s \geq \frac{3}{2}+\epsilon$ with $\epsilon \geq \frac{10+t}{2(k-t-2)}$, the growth of $\sum_{j=0}^{k-\ell} \binom{n}{j}$ is higher than $\sum_{\ell=t+2}^{k} \ell^2(k-t+1)^{\ell-t-1} \sum_{i=t}^{\ell} \binom{\ell}{i}$. 


Now, comparing the growths of both sides of Eq. (\ref{threshold-eqn}) we obtain 
\begin{equation}\label{growth-eqn}
	\mathcal{O}(t^4) \cdot k \cdot \Theta(n^{k-t-2}) \leq \mathcal{O}(t^2) \cdot \Theta(n^{k-t-1}).
\end{equation}



\textbf{Case 1.} If $t=\mathcal{O}(1)$, then the LHS of Eq. (\ref{growth-eqn}) is $k \cdot \Theta(n^{k-t-2})$, where as the RHS of Eq. (\ref{growth-eqn}) is $\Theta(n^{k-t-1})$. Thus by setting $n=\Theta(k^{s^\prime})$ for some $s^\prime$, we see that LHS has growth $k^{s^\prime(k-t-2)+1}$ and RHS has growth $k^{s^\prime(k-t-1)}$. Thus, Eq. (\ref{growth-eqn}) holds for $n \geq k^{\max\{3/2+\epsilon, 1\}} = k^{3/2+\epsilon}$. 


\textbf{Case 2.} If $t=\Theta(k)$, then the LHS of Eq. (\ref{growth-eqn}) has growth $\Theta(k^5) \cdot \Theta(n^{k-t-2})$, where as the RHS of Eq. (\ref{growth-eqn}) has growth $\Theta(k^2) \cdot  \Theta(n^{k-t-1})$. Thus by setting $n=\Theta(k^{s^\prime})$ for some $s^\prime$, we see that LHS has growth $k^{s^\prime(k-t-2)+5}$ and RHS has growth $k^{s^\prime(k-t-1)+2}$. Thus, Eq. (\ref{growth-eqn}) holds for $n \geq k^{\max\{3/2+\epsilon, 3\}}$. As $\epsilon \geq \frac{10+t}{2(k-t-2)}$, for $t < \frac{3}{4}k-4$ we have $\epsilon < \frac{3}{2}$, which further implies $\frac{3}{2}+\epsilon < 3$. Hence, Eq. (\ref{growth-eqn}) holds for $n \geq c \cdot k^3$.


\textbf{Case 3.} If $t = o(k)$, then we have two cases.

\textbf{Subcase 1.} Let $t=\Theta(k^{\epsilon^\prime})$ for some $\epsilon^\prime$ with $\epsilon^\prime \leq \frac{1}{4}$. Then $1+2\epsilon^\prime \leq \frac{3}{2}+\epsilon$ as $\epsilon > 0$. In that case Eq. (\ref{growth-eqn}) is equivalent to $\mathcal{O}(t^2) \cdot \mathcal{O}(k) = \mathcal{O}(k^{1+2\epsilon^\prime}) \leq \Theta(n)$, which holds when $n = \Theta(k^{s^\prime})$ with $s^\prime \geq \max\{\frac{3}{2}+\epsilon, 1+2\epsilon^\prime\} = \frac{3}{2}+\epsilon$. 


\textbf{Subcase 2.} Let $t=\Theta(k^{\epsilon^\prime})$ for some $\epsilon^\prime$ with $\epsilon^\prime = \frac{1}{4}+o(1)$. Then for sufficiently large $k$ we have $1+2\epsilon^\prime \geq \frac{3}{2}+\epsilon$ as $\epsilon^\prime \geq \frac{1}{4} + \frac{1}{2} \cdot \frac{10+\Theta(k^{\epsilon^\prime})}{2(k-\Theta(k^{\epsilon^\prime})-2)}$. In that case Eq. (\ref{growth-eqn}) is equivalent to $\mathcal{O}(t^2) \cdot \mathcal{O}(k) = \mathcal{O}(k^{1+2\epsilon^\prime}) \leq \Theta(n)$, which holds when $n = \Theta(k^{s^\prime})$ with $s^\prime \geq \max\{\frac{3}{2}+\epsilon, 1+2\epsilon^\prime\} = 1+2\epsilon^\prime$.

This completes the proof of the theorem. 
\end{proof}


    \section*{Acknowledgment}
    The authors wish to thank the Institute of Mathematical Sciences, Chennai for the financial support received through the institute postdoctoral program.  The authors would like to thank Dr. Stijn Cambie for his valuable suggestions on the previous draft of this manuscript.

	\bibliographystyle{amsplain}

\end{document}